\UseRawInputEncoding
\documentclass{article}

\usepackage{amsmath,amsthm,amsfonts,amssymb}
\usepackage{graphicx}
\usepackage[normalem]{ulem}
\usepackage{authblk}

\newtheorem{lm}{Lemma}[section]
\newtheorem{prop}[lm]{Proposition}
\newtheorem{theorem}[lm]{Theorem}

\newtheorem{cy}[lm]{Corollary}

\theoremstyle{definition}
\newtheorem{df}[lm]{Definition}

\newtheorem{rk}[lm]{Remark}

\usepackage{color}
\usepackage[dvipsnames]{xcolor}

\newcommand{\beq}{\begin{equation}}
\newcommand{\eeq}{\end{equation}}
\newcommand{\be}{\begin{enumerate}}
\newcommand{\ee}{\end{enumerate}}
\newcommand{\bp}{\begin{proof}}
\newcommand{\ep}{\end{proof}}
\newcommand{\bi}{\begin{itemize}}
\newcommand{\ei}{\end{itemize}}
\newcommand{\bea}{\begin{eqnarray*}}
\newcommand{\eea}{\end{eqnarray*}}
\newcommand{\bml}{\begin{multline*}}
\newcommand{\eml}{\end{multline*}}

\newcommand{\ang}[1]{\left\langle #1 \right\rangle}
\newcommand{\ZZ}{\mathbb{Z}}

\DeclareMathOperator{\Th}{Th}

\edef\storedcatcodeat{\the\catcode`\@} \catcode`\@=11

\catcode`\@=\storedcatcodeat\relax

\begin{document}

\renewcommand{\thefootnote}{\fnsymbol{footnote}}
\renewcommand{\thefootnote}{\arabic{footnote}}

\title{On countable elementary free groups}
\author{
    Olga Kharlampovich\thanks{
        City University of New York, Graduate Center and Hunter College\\
        \indent\hspace{0.25cm}\texttt{okharlampovich@gmail.com}
    }
    \ and
    Christopher Natoli\thanks{
        City University of New York, Graduate Center\\
        \indent\hspace{0.25cm}\texttt{chrisnatoli@gmail.com}
    }
}
\date{}

\maketitle

\begin{abstract}
    We prove that if a countable group is elementarily equivalent to a non-abelian free group and all of its  abelian subgroups are cyclic, then the group is a union of a chain of regular NTQ groups (i.e., hyperbolic towers).
\end{abstract}

\section{Introduction}

A full characterization of finitely generated groups that are elementarily equivalent to non-abelian free groups was given in \cite{KhaMya2006elementary}, \cite{Sel2006diophantine}. These constructions are called hyperbolic towers or, equivalently, regular NTQ groups. Groups that are elementarily equivalent to non-abelian free groups have since been deemed \emph{elementary free groups}. Elementary free groups that are not finitely generated are not as well understood.

Ultraproducts are natural examples of elementary free groups, but they are uncountable and less tractable from a group-theoretic point of view. Some progress has been made in the direction of understanding countable elementary free groups by Kharlampovich, Myasnikov, and Sklinos \cite{KhaMyaSkl2020fraisse} by modifying the model-theoretic construction of Fra\"iss\'e limits, which yielded a countable elementary free group that is the union of a chain of finitely generated elementary free groups, in addition to other nice model-theoretic properties among which it is unique. They ask whether every countable elementary free group can be obtained as the union of a chain of finitely generated elementary free groups. \cite{KhaNat2020non} give $\ZZ*(\ZZ\oplus\mathbb{Q})$ as a counter-example, employing the fact that finitely generated elementary free groups cannot contain non-cyclic abelian subgroups.

In this paper, we positively answer a modified version of the question from \cite{KhaMyaSkl2020fraisse}, in which we add the extra condition that abelian subgroups must be cyclic. Our main theorem is as follows:
\begin{theorem} \label{descr} Let $M$ be a countable elementary free group in which all  abelian subgroups are cyclic. Then $M$ is a union of a chain of regular NTQ groups (i.e., hyperbolic towers).
\end{theorem}
\noindent We prove this theorem in two cases: when the language is that of pure groups and when the language is expanded by adding generators for a fixed non-abelian free group $F\le M$.

\section{Preliminaries}

We use two special graphs of groups constructions, both involving surfaces: hyperbolic towers and JSJ decompositions. In the discussions below, note that a \emph{maximal boundary subgroup} of a compact connected surface $\Sigma$ with non-empty boundary is the cyclic fundamental group of a boundary component of $\Sigma$.

\begin{df}\label{def:surface-tp-vertex}
Let $\Gamma$ be a graph of groups with fundamental group $G$.
Then a vertex $v\in \Gamma$ is called a \emph{surface type vertex} if the following conditions hold:
\begin{itemize}
\item the vertex group $G_v=\pi_1(\Sigma)$ for a connected compact surface $\Sigma$ with non-empty boundary such that either the Euler characteristic of $\Sigma$ is at most $-2$ or $\Sigma$ is a once punctured torus;
\item For every edge $e\in \Gamma$ adjacent to $v$, the edge group $G_e$ embeds onto a maximal boundary subgroup of $\pi_1(\Sigma)$, and this induces a one-to-one correspondence between the set of edges adjacent to $v$ and the set of boundary components of $\Sigma$.
\end{itemize}
\end{df}

\begin{df}
Let $G$ be a group and $H$ be a subgroup of $G$. Then $G$ is a \emph{hyperbolic floor} over $H$
if $G$ admits a graph of groups decomposition $\Gamma$ where the set of vertices can be partitioned in two subsets $V_S,V_R$
such that
\begin{itemize}
\item each vertex in $V_S$ is a surface type vertex;
\item $\Gamma$ is a bipartite graph between $V_S$ and $V_R$;
\item the subgroup $H$ of $G$ is the free product of the vertex groups in $V_R$;
\item either there exists a retraction $r:G\to H$ that, for each $v\in V_S$, sends $G_v$ to a non-abelian image or $H$ is cyclic and there exists a retraction $r': G * \ZZ \to H * \ZZ$ that, for each $v\in V_S$, sends $G_v$ to a non-abelian image.
\end{itemize}

\end{df}

A hyperbolic tower is a sequence of hyperbolic floors and free products with finite-rank free groups and closed surface  groups of Euler characteristic at most $-2.$

\begin{df}
A group $G$ is a \emph{hyperbolic tower} over a subgroup $H$ if there exists a sequence $G=G^m>G^{m-1}>\ldots>G^0=H$ such that for each $i$, $0\leq i<m$,
one of the following holds:
\begin{itemize}
\item the group $G^{i+1}$ has the structure of a hyperbolic floor over $G^i$, in which $H$ is contained in one of the vertex groups that generate $G_i$ in the
 floor decomposition of $G^{i+1}$ over $G^i$;
\item the group $G^{i+1}$ is a free product of $G^{i}$ with a finite-rank free group or the fundamental group of a compact surface without boundary of Euler characteristic at most $-2$.
\end{itemize}
\end{df}

Note that a hyperbolic tower over a free group is equivalent to the class of groups known as \emph{regular NTQ groups} \cite{KhaMya2006elementary}. Hyperbolic towers completely characterize finitely generated elementary free groups as well as elementary embeddings among torsion-free hyperbolic groups.

\begin{theorem}[Kharlampovich-Myasnikov, Sela]
A finitely generated group $G$ is elementarily equivalent to a nonabelian free group $F$ if and only if $G$ is a nonabelian hyperbolic tower over $F$.
\end{theorem}

\begin{theorem}[Kharlampovich-Myasnikov, Sela]
Let $G$ be torsion-free hyperbolic. If $G$ is a hyperbolic tower over a nonabelian subgroup $H$, then $H$ is an elementary subgroup of $G$. \end{theorem}

\begin{theorem}\cite{Per2011elementary} 
Let $G$ be a torsion-free hyperbolic group. If $H$ be an elementary subgroup of $G$,
then $G$ is a hyperbolic tower over $H$.
\end{theorem}

We now turn to JSJ decompositions, recalling definitions of Dehn twists and canonical automorphisms as well.

A \emph{one-edge cyclic splitting} of a group $G$ is a graph of groups with fundamental group $G$ that is either a segment (i.e., amalgamated product) or a loop (i.e., HNN extension), with the single edge group being infinite cyclic.
Given a graph $\Lambda$ of groups with fundamental group $G$, $H\le G$ is \emph{elliptic} if it is contained in a conjugate of a vertex group in $\Lambda$; analogously for an element of $G$.

\begin{df}
Let $\Lambda$ be a one-edge cyclic splitting of $G$ into either $A*_CB$ or $A*_C$, and let $\gamma\in C_G(C)$. The \emph{Dehn twist} of $\Lambda$ by $\gamma$ is the automorphism that restricts to identity on $A$ and to conjugation on $B$ or, respectively, restricts to identity on $A$ and sends the stable letter $t$ to $t\gamma$. A \emph{canonical automorphism} of $G$ is an automorphism generated by Dehn twists and inner automorphisms. In the case of $H$-automorphisms for some subgroup $H\le G$, canonical $H$-automorphisms must fix $H$ pointwise.
\end{df}

Let $\Lambda$ be a graph of groups with fundamental group $G$. A vertex of $\Lambda$ is \emph{quadratically hanging} or \emph{QH} if the corresponding vertex group is a fundamental group of a closed hyperbolic surface with boundary, each incident edge group is a finite-index subgroup of a maximal boundary subgroup, and, conversely, every maximal boundary subgroup contains an incident edge group as a finite-index subgroup. Note that this is weaker than Definition \ref{def:surface-tp-vertex} of surface type vertices.

\begin{df}
A cyclic \emph{JSJ decomposition} of a group $G$ is a graph of groups $\Lambda$ with fundamental group $G$ that is maximal in the sense that all possible one-edge cyclic splittings of $G$ are either splittings along edges in $\Lambda$ or can be found by splitting a QH vertex along a simple closed curve. A graph of groups $\Lambda$ with fundamental group $G$ is a \emph{JSJ decomposition of $G$ relative to $H\le G$} if $H$ is elliptic in a non-QH vertex of $\Lambda$ and $\Lambda$ is maximal only with respect to one-edge cyclic splittings in which $H$ is elliptic.
\end{df}

Note that any pair of noncompatible (intersecting)  splittings of $G$ comes from a pair of intersecting simple closed curves on one of the corresponding surfaces in the  JSJ decomposition of $G$.

We review some equivalent definitions of limit groups that we need for our purposes.
\begin{df}
    Suppose $G$ is a finitely generated group. Then $G$ is a \emph{limit group} if it satisfies one (i.e., all) of the following equivalent properties:
    \begin{itemize}
        \item $G$ is freely discriminated (i.e., given finitely many non-trivial elements $g_1,\ldots,g_n\in G$, there exists a homomorphism $\phi$ from $G$ to a free group such that $\phi(g_i)\ne1$ for $i=1,\ldots,n$);
        \item $G$ is universally equivalent to a non-abelian free group in the language without constants (i.e., $\Th_\forall(G)=\Th_\forall(F)$);
        \item $G$ is the coordinate group of an irreducible variety over a free group (see
\cite{KhaMya2006elementary} or \cite{KhaMya2010equations} for definitions).
    \end{itemize}
\end{df}
\begin{df}\label{def:res-lim-gp}
    Suppose $G$ is a finitely generated group containing a non-abelian free group $F\le G$. Then $G$ is a \emph{restricted limit group} if it satisfies one (i.e., all) of the following equivalent properties:
    \begin{itemize}
        \item $G$ is $F$-discriminated by $F$ (i.e., given finitely many non-trivial elements $g_1,\ldots,g_n\in G$, there exists an $F$-homomorphism $\phi:G\to F$ such that $\phi(g_i)\ne1$ for $i=1,\ldots,n$);
        \item $G$ is universally equivalent to $F$ in the language with $F$ as constants;
        \item $G$ is the coordinate group of an irreducible variety over $F$.
    \end{itemize}
\end{df}
\noindent We refer the reader to \cite{KhaMya2010equations} for more equivalent definitions of limit groups and a proof of their equivalence.


\section{JSJ decompositions of subgroups of an elementary free group without non-cyclic abelian subgroups}\label{sec:jsj-of-subgps}

Let $F$ be a finitely generated free group. We can consider the theory $\Th(F)$  with or without coefficients, i.e., with or without the generators of $F$ included as constants in the language. In the case where we have constants, all the groups elementarily equivalent to $F$ contain a designated copy of $F$.
Below, in the case of a coefficient-free system $S(X)=1$ we put $F_{R(S)}=F(X)/R(S)$ and in the case when there are coefficients $F_{R(S)}=(F\ast F(X))/R(S)$.

Let $M$ be a countably generated elementary free group in which all abelian subgroups are cyclic. Suppose the theory has coefficients. Then $F\leq M$. Let $G$ be a subgroup of $M$ generated by $F$ and finitely many elements. We will show (Corollary \ref{qh}) that for every free factor in the free decomposition of $G$ relative to $F$, the JSJ decomposition of such a factor is either trivial or has a QH subgroup.

Note first that in the language with constants, we have $\Th_\forall(F)\supseteq\Th_\forall(G)\supseteq\Th_\forall(M)=\Th_\forall(F)$, so by Definition \ref{def:res-lim-gp}, $G$ is a limit group and so there exists an irreducible system of equations $S(X)=1$ over $F$ such that $G=F_{R(S)}$. In the case without constants, $\Th_\forall(G)\supseteq\Th_\forall(M)$. In particular, for any reduced word $w(x,y)$, $M\models\forall x\forall y([x,y]\ne1\rightarrow w(x,y)\ne1)$, hence so does $G$. So for any non-commuting $a,b\in G$, $\ang{a,b}$ is free, hence $\Th_\forall(\ang{a,b})\supseteq\Th_\forall(G)\supseteq\Th_\forall(M)=\Th_\forall(\ang{a,b})$, so again we can write $G$ as the coordinate group of an irreducible variety. The group $G$ does not have non-cyclic abelian subgroups.

In the free decomposition of $G$ relative to $F$, one free factor contains $F$ and the others have trivial intersection with $F$. The factors that are isomorphic to closed surface groups and cyclic groups have trivial JSJ decompositions. We assume that at least one factor is not isomorphic to a surface group. We first wish to understand the JSJ decomposition of such a factor. Since the same argument gives a proof for both cases, when $F\leq G$ and when their intersection is trivial, we may assume that $G$ is freely indecomposable relative to $F$ and consider only one case, say $F \leq G$.

Then $G$ has a non-degenerate cyclic JSJ decomposition \cite{KhaMya2005effective} relative to $F$; denote it by $D$.
Let $B$ be a basis of $F$. Then $G=\ang{B,X\mid S}$ gives the canonical finite presentation of $G$ as the fundamental group of $D$.
Let $E_{r}$ be the set of edges between rigid subgroups. We will assume that the edge groups corresponding to edges in $E_r$ are maximal cyclic in $M$. If not, we add the roots of the maximal cyclic subgroups in $M$ to $G$, denote the new group by $\bar G$ and replace $D$ by the cyclic JSJ decomposition of $\bar G$.  Note that all abelian subgroups of $M$ are cyclic, which guarantees that there is a deepest root to add to $G$.  Moreover,  since $G$ inherits the splitting from the JSJ decomposition of $\bar G$, we do not have to add more roots to $\bar G$.

Let $A_E$ be the group of $F$-automorphisms (or simply automorphisms, in the case where $F \cap G=1$) of $G$ generated by Dehn twists along the edges of $D$. The group $A_E$ is abelian by \cite{Levitt}. Recall that two solutions $\phi_1$ and $\phi_2$ of the equation $S(X) = 1$ are {\em $A_E$-equivalent} if there is an automorphism $\sigma \in A_E$  such that $\phi_1\sigma = \phi_2$.

Recall that if $A$ is a group of canonical automorphisms of $G$, then the maximal standard quotient of $G$ with respect to $A$ is defined as the quotient $G/R_A$ of $G$ by the intersection $R_A$ of the kernels of all solutions of $S(X) = 1$ that are minimal with respect to $A$ (see \cite{KhaMya2005effective}  for details). Minimality is taken with respect to the length of a solution, where the length of a solution $\psi$ is defined as $|\psi|=\min_{h\in F}\max_{x\in X}|h\psi (x)h^{-1}|.$

By \cite[Theorem 9.1]{KhaMya2005effective} the maximal standard quotient
$G/R_{A_D}$ of $G$ with respect to the whole group of canonical
automorphisms $A_D$ is a proper quotient of $G$, i.e., there
exists a disjunction of systems of equations  (it is equivalent to one equation if we consider the theory with coefficients) $V(X) = 1$ such that $V \not \in R(S)$ and  all
minimal solutions of $S(X)=1$ with respect to the canonical group
of automorphisms $A_D$ satisfy  $V(X)=1.$  Notice that there is a finite family of maximal limit groups $L_1,\ldots , L_k$  such that every homomorphism $G/R_{A_D}\rightarrow F$ factors through one of them.  Now, compare
this with the following result.

\begin{lm} \label{le:max-stand-AE}
 The maximal standard quotient of $G$
 with respect to the group $A_E$ is equal to $G$, i.e.,
 the set of minimal solutions with respect
 to $A_E$ discriminates $G$.
 \end{lm}

\begin{proof} Suppose, to the contrary, that the maximal standard
quotient $G/R_{A_E}$ is a proper quotient of $G$, i.e., there exists $V \in G$ such that $V \neq 1$ and $V^\phi = 1$ for any solution $\phi$ of $S$ minimal with respect to $A_E$. Recall that the group $A_E$ is generated by Dehn twists along the edges of $D$.  If $c_e$ is a given generator of the cyclic subgroup associated with the edge $e$, then we know how the Dehn twist $\sigma$ associated with $e$ acts on the generators from the set $X$. Namely, if $x \in X$ is a generator of a vertex group, then either $x^\sigma = x$ or $x^\sigma = c_e^{-1}xc_e$. Similarly, if $x \in X$ is a stable letter then either $x^\sigma = x$ or $x^\sigma = xc_e$. It follows that for $x \in X$ one has $x^{\sigma^n} = x$ or $x^{\sigma^n} = c_e^{-n}xc_e^n$ or $x^{\sigma^n} = xc_e^n$ for every $n \in \mathbb{Z}$. Now, since the centralizer of $c_e$ in $G$ is cyclic the following equivalence holds:
$$ \exists n \in \mathbb{Z} (x^{\sigma^n}  = z)
\Longleftrightarrow \left \{ \begin{array}{ll}
                            x= z & \mbox{if}\; x^\sigma  = x\\
                            \exists y ([y,c_e] = 1 \wedge y^{-1}xy = z) & \mbox{if}\; x^\sigma  = c_e^{-1}x c_e\\
                            \exists y([y,c_e]=1\wedge xy=z) & \mbox{if}\; x^\sigma=xc_e
                            \end{array} \right.
                            .$$

Similarly, since the group $A_E$ is finitely generated abelian one can write down a formula which describes the relation $$\exists \alpha \in A_E (x^{\alpha}  = z).$$ One can write the elements $c_e$ as words in the generators $X$, say $c_e = c_e(X)$. Now the sentence
$$\forall X\exists Y \exists Z \left(S(X) = 1 \rightarrow  \left(\bigwedge_{i = 1}^{|Y|}[y_i,c_i(X)]=1 \wedge
Z = X^{\sigma_Y} \wedge V(Z)=1\right)\right)$$
holds in the group $F$. Indeed, this formula tells one  that each solution of $S(X)=1$ is $A_E$-equivalent to a (minimal) solution $Z$ that satisfies the equation $V(Z)=1$. Hence this formula is true in $M$; in particular, it is true for $X\subset G$. The subgroup generated by $Z$ is isomorphic to the subgroup generated by $X$ in $M$, and $X$ and $Z$ satisfy the same relations.  But $V(X)\ne1$, so the formula cannot be true in $M$: contradiction.
\end{proof}

\begin{lm}\label{le:exist-QH}
    There exist QH subgroups  in $D$.
\end{lm}

\begin{proof}
By \cite[Theorem 9.1]{KhaMya2005effective}, $G/R_{A_D}$ is a proper quotient of $G$,
 but by Lemma~\ref{le:max-stand-AE}, $G/R_{A_E}$ is not proper, so we have $A_D \neq A_E$, hence 
  $D$ has QH subgroups.
\end{proof}
\begin{cy} \label{qh} A JSJ decomposition of a finitely generated freely indecomposable subgroup of $M$  is either trivial or has a QH subgroup.
\end{cy}

\section{Proof of the main theorem}

 We continue the notation ($F$, $M$, $G$, $D$, etc.) from Section \ref{sec:jsj-of-subgps}. Let $K$ be the fundamental group of the graph of groups obtained from $D$ by removing all QH subgroups.

\begin{lm} \label{1.4} There is a $K$-homomorphism from $G$ into M with a non-trivial kernel containing $R_{A_D}.$

The quotient is also a quotient of one of the  limit groups $L_1,\ldots , L_k$ that are maximal limit quotients of $G/R_{A_D}$.
\end{lm}

\begin{proof}
The generators $X$ in $G$ corresponding to the decomposition $D$ can be partitioned as $X=X_1\cup X_2$  where $B\cup X_2$ generate $K$. Since $G/R_{A_D}$ is a proper quotient,  the system of equations $V(X_1,X_2)=1$ defining $R_{A_D}$  is not in $R(S)$.

Any solution $\psi:G\to F$ can be precomposed by an automorphism $\sigma\in A_D$ such that $\psi\sigma$ solves $V$. Since $A_E\trianglelefteq A_D$, we can factor $\sigma=\gamma\beta$ where $\beta\in A_E$ and $\gamma\in A_Q$. Note $A_E$ is abelian and $\gamma|_K$ only conjugates the free factors of $K$. The existence of such an automorphism $\sigma$ for any solution $\psi$ is expressed by the following sentence:
\begin{multline*}
    \forall X_1\forall X_2\exists Y_1\exists Y_2\exists T \ S(X_1,X_2)=1\\
\rightarrow \left(\bigwedge_{i=1}^{|T|}[t_i,c_i(X_1,X_2)]=1 \wedge Y_2=X_2^{\sigma_T}\wedge S(Y_1,Y_2)=1\wedge V(Y_1,Y_2)=1\right).
\end{multline*}
It is true in $F$, and therefore also in $M$. If we take $X_1,X_2$ to be the generators of $G$ and denote the witnesses to $Y_1,Y_2,T$ in $M$ by $Y_1,Y_2,T$, then $M$ models
$$\bigwedge_{i=1}^{|T|}[t_i,c_i(X_1,X_2)]=1 \wedge Y_2=X_2^{\sigma_T}\wedge S(Y_1,Y_2)=1\wedge V(Y_1,Y_2)=1.$$

Let $G_1=\ang{G,Y_1,Y_2}$ and define a homomorphism $\phi:G\to G_1$ by $\phi(X_1)=Y_1$ and $\phi(X_2)=Y_2$, i.e., $\phi|_{K}=\gamma\beta|_{K}$. The subgroup of $M$ generated by $Y_2$ is isomorphic to $K$. Let $\phi_1=\phi\beta^{-1}$. Then $\phi_1|_{K}=\gamma|_K$ acts by conjugation on the free factors of $K$, i.e., if $K=K_1*\cdots*K_s$ is a free decomposition of $K$, then for each $i$, there exists $g_i\in G$ such that for all $k\in K_i$, $\phi_1(k)=k^{g_i}$.

Let $\phi_1(G)=P_1*\cdots*P_m$ be a free decomposition of $\phi_1(G)$.
Define a map $\nu$ on $\phi_1(G)$ as follows: if $P_j$ contains $K_i^{g_i}$ for some $i$, let $\nu(\ell)=\ell^{g_i^{-1}}$ for all $\ell\in P_j$; otherwise, $\nu$ acts identically on $P_j$.
Also, define $\eta$ on $\nu\phi_1(G)$ by $\eta\nu(P_j)=1$ if $P_j$ does not contain any factor $K_i$ and $\eta\nu(P_j)=P_j$ otherwise.
Then $\phi_2=\eta\nu\phi_1$ is the desired homomorphism.
\end{proof}

\begin{rk}
We can also view $\phi_2$ from a geometric perspective as follows. For each QH subgroup $Q$  in $D$ we associate a possibly trivial collection of simple closed curves on the corresponding surface that are mapped to the identity by $\phi _2$ as follows. The image of $Q$ inherits a free decomposition from $\phi _1(G)$.  This free decomposition lifts to a cyclic decomposition of $Q$. The curves correspond to the edges  of this cyclic decomposition. We construct another graph of groups $\Gamma$ obtained from $D$  by cutting the surfaces corresponding to QH vertex groups along this collection of curves, filling each curve with a disk and removing the surfaces that are not connected to any boundary components of the original QH vertex subgroups. Let $\bar G$ be the fundamental group of $\Gamma$. There is a natural epimorphism $\mu : G\rightarrow \bar G$.  We define $\tau: \bar G\rightarrow G_1$ on each free factor $\bar G_i$ of $\bar G$ that is associated with a connected component of $\Gamma$ by $\tau(\bar G_i)=\phi_2\mu^{-1}(\bar G_i)$. Then $\phi_2=\tau\mu$. Notice that no non-trivial simple closed curve on a QH vertex group of $\Gamma$ is mapped to the trivial element in $G_1$.
\end{rk}

We need the following lemma for the proof of Theorem \ref{main}.

\begin{lm} \label{NTQ}  There exist a fundamental sequences  $$G\xrightarrow{\pi_1}G_1\xrightarrow{\pi _2}\cdots\xrightarrow{\pi _n}G_n=F$$ in the canonical Hom-diagram for $G$ (see \cite{KhaMya2012quantifier}) such that for each $i=1,\ldots ,n$ there is a subgroup $H_i<M$ such that
\begin{enumerate}
\item $H_i$ is a quotient of $G_i$;
\item the JSJ decomposition $\Delta _i$ of $H_i$ (if $H_i$ is a free product, then $\Delta _i$ is a Grushko decomposition followed by JSJ decompositions of the free factors) contains a QH subgroup;
\item if $K_i$ is the fundamental group of the subgraph of groups of $\Delta _i$ containing all rigid vertex groups, then $K_i\leq H_{i+1}$;
\item there is a proper $K_i$-homomorphism $\beta _{i+1}: H_i\rightarrow H_{i+1}$  such that its kernel contains $R_{A_{\Delta _i}}$;
\item $H_{i}$ is freely indecomposable relative to  the isomorphic image of $K_{i-1}$ in $H_i$.  Each free factor of $H_i$ contains a conjugate of a free factor of $K_{i-1}$.
\end{enumerate}
If \begin{equation}\label{eq:1} G\xrightarrow{\beta_1} H_1\xrightarrow{\beta _2}\cdots \xrightarrow{\beta _n}H_n=F.\end{equation} is the corresponding fundamental sequence describing some homomorphisms from $G$ to a free group, then the NTQ group $N$ for this fundamental sequence is a regular NTQ group. Moreover, $M$ contains a quotient of $N$ containing $G$.
\end{lm}

\begin{proof} We set $\beta _1=\phi _2$ and $H_1=\phi _2(G)$. Since every chain of proper limit quotients of a limit group is finite, the statement of the lemma can be proved using Lemma \ref{1.4} inductively.
\end{proof}

\begin{theorem}  \label{main} For every finitely generated subgroup $G$ of $M$ (in the case with coefficients, $G$ must contain $F$)  there exists a regular NTQ group that is a subgroup of $M$ and contains $G$.
\end{theorem}
The proof of the theorem will be given in the next section. To give an idea,  notice  that by Lemma \ref{NTQ}, $M$ contains a quotient of the regular NTQ group $N$ such that this quotient contains $G$.   Denote this quotient by $G_1$. If $G_1$  is isomorphic to $N$ the theorem is proved. Otherwise we will construct a regular NTQ group $\bar N_1$  containing  $G_1$  and such that $M$ contains a quotient $G_2$ of $\bar N_1$, moreover $G_1\leq G_2$. If $G_2=\bar N_1$, the theorem is proved. Otherwise, we construct a regular NTQ group $\bar N_2$ such that $G_2\leq \bar N_2$ and $M$ contains a quotient $G_3$ of $\bar N_2$.  We will show that the construction can be designed such a way that it stops, namely on some step $i$, $G_{i+1}=\bar N_{i}.$

\begin{proof}[Proof of Theorem \ref{descr}]
Let $G$ be the subgroup of $M$ generated by $F$ and the first $n$ generators of $M$.  By Theorem \ref{main} there is  a regular NTQ group that is a subgroup of $M$ and contains $G$.  Denote this NTQ group by $F_1$.
It is the NTQ group over the base group $F$.
Take now the subgroup $\Gamma$ of $M$ generated by $F_1$ and next $n$ generators of $M$.  There is  a regular NTQ group $F_2$ over $F_1$ that is a subgroup of $M$ and contains $\Gamma$.

Suppose, by induction, that  a regular NTQ group $F_i\leq M$ containing the first $ni$ generators of $M$ has been already constructed. Let $\Gamma _i$ be the subgroup of $M$ containing $F_i$ and the first $n(i+1)$ generators of $M$. We can construct  a regular NTQ group $F_{i+1}$ over $F_i$ containing $\Gamma _i$. Then $M$ is the union of the  chain $F\prec F_1 <\ldots<F_i<\ldots $.
\end{proof}

Similarly to Theorem \ref{descr}, one can show the following using a completely analogous proof:
\begin{theorem} \label{descr1} Let $H$ be a non-cyclic torsion-free hyperbolic group. Let $M\ge H$ be a countable group that is elementarily equivalent to $H$ in the language with constants for the generators of $H$, and suppose all finitely generated abelian subgroups of $M$ are cyclic. Then $M$ is a union of a chain of hyperbolic towers over $H$.
\end{theorem}

\section{Proof of Theorem \ref{main}}
In this section  we will prove Theorem \ref{main}.
The  proof is, basically, the adaptation for our needs of the construction
of the $\forall\exists$-tree from \cite{KhaMya2006elementary} or \cite{KhaMya2012quantifier}.
It is written by the first author.

{\bf The strategy of the proof.}

In the free decomposition of $G$ relative to $F$, one free factor contains $F$ and the others have trivial intersection with $F$. If all factors are isomorphic to closed surface groups and cyclic groups, we are done because in both cases $G$ is a regular NTQ group as required. Therefore we assume that at least one factor is not isomorphic to a surface group.

 Let $G=F_{R(U_0)}$ be a limit group generated by elements $X$ that is a subgroup of the elementary free group $M$.  We will use the following strategy. Let $N$ be the  regular NTQ group from Lemma \ref{NTQ}. If $N\leq M$ the theorem is proved. Suppose that $N\not \leq M$.   Then $M$ contains  the  proper quotient of  $N$ corresponding to the system of equations $U_1=1$,  and $G_1=F_{R(U_1)}$ contains $G$. Denote  $N=N_0=\bar N_0$. We say that $N_0, \bar N_0$ are constructed on the {\bf initial step}.

\begin{df} We call  fundamental sequences satisfying the first and second restrictions introduced in \cite{KhaMya2006elementary} ( see also \cite{KhaMya2012quantifier}) {\em well-aligned} fundamental sequences.\end{df}

 We perform the {\bf first step} and construct the so called block-NTQ groups $N_1$ describing all homomorphisms $G\rightarrow F$ that factor through $\bar N_0$, satisfy $U_1=1$ and belong to well-aligned fundamental sequences. Such block-NTQ group $N_1$ contains as a subgroup a regular NTQ group $\bar N_1$ containing $G$.
Since this can be described by first order sentences, we  conclude the same way as we did in Lemma \ref{NTQ} that $M$ contains a quotient of at least one of the groups $\bar N_1$ containing $G$.
If this quotient is isomorphic to $\bar N_1$ we stop, and the theorem is proved.

Otherwise $G_2=F_{R(U_2)}$ is a proper quotient of $\bar N_1$ , we go to the {\bf second step}
and construct refined block-NTQ systems $N_2$ describing all homomorphisms $G\rightarrow F$ that factor through $N_1$ and $\bar N_1$, satisfy $U_2=1$ and belong to well-aligned fundamental sequences.

We will prove that this construction stops. The construction of the ``refined'' NTQ systems is quite complicated but it  is very similar to the construction of a branch of the $\exists\forall$-tree in
the proof of the decidability of the theory of a free group \cite{KhaMya2006elementary}, \cite{KhaMya2012quantifier}.

We will represent the construction as a path. We assign to each vertex of the path some set of homomorphisms $G\rightarrow F$ and a regular NTQ group (hyperbolic tower) containing $G$ as a subgroup.
We assign the set of all homomorphisms  $G\rightarrow F$ to the initial vertex $w_{0}$.

If $Q$ is a QH subgroup, let $size (Q)=(genus (Q), |\chi (Q)|).$  Pairs are ordered left lexicographically. Let  complexity of the JSJ decomposition be the tuple of complexities
$$(size (Q_1),\ldots ,size (Q_k)),$$
where $size (Q_1)\geq \ldots \geq size (Q_k).$ We order tuples left lexicographically.
Similarly the {\em regular size} of an NTQ group is defined.

It is convenient to define  the notion of complexity of a fundamental sequence ($Cmplx (Var_{fund})$) at follows:
$$Cmplx(Var_{\rm fund}) =$$ $$(dim (Var_{\rm fund}) + factors (Var_{\rm fund}), (size (Q_1),\ldots ,size (Q_m))),$$
where $factors (Var_{\rm fund})$ is the number of freely indecomposable, non-cyclic terminal factors (this component only appears in fundamental sequences relative to  subgroups),  and   $(size (Q_1),\ldots, size (Q_m))$ is the regular size of the system. The complexity is a tuple of numbers which we compare in the left lexicographic order.  The number $dim (Var_{\rm fund}) + factors (Var_{\rm fund})$ is called the Kurosh rank of the fundamental sequence.

Let $K$ be a finitely generated group. Recall that any family of homomorphisms $\Psi=\{\psi _i:K\rightarrow F \}$ factors through a finite set of maximal fully residually free groups $H_1,\ldots , H_k$ that all are quotients  of $K$.  We first take a quotient $K_1$ of $K$ by the intersection of the kernels of all homomorphisms from $\Psi$, and then construct maximal fully residually free quotients $H_1,\ldots , H_k$ of $K_1$. We say that $\Psi$ {\em discriminates} groups  $H_1,\ldots , H_k$, and that each $H_i$ is {\em a fully residually free group discriminated by $\Psi .$}

Let $S$ be a compact surface. Given a homomorphism $\phi :\pi _1(S)\rightarrow H$
a family of pinched curves is a collection $\mathcal C$ of disjoint, non-parallel, two-sided
simple closed curves, none of which is null-homotopic, such that the fundamental
group of each curve is contained in ker $\phi$ (the curves may be parallel to a boundary component).
The map $\phi$  is {\em non-pinching} if there is no pinched curve: $\phi$ is injective in restriction to
the fundamental group of any simple closed curve which is not null-homotopic.

\subsection{Initial step} We can
construct algorithmically a finite number of NTQ systems  corresponding to branches $b$ of the canonical $Hom$-diagram described in \cite[Section 2.2]{KhaMya2012quantifier}. For at least one of them we can construct the well aligned fundamental sequence (\ref{eq:1}) and the regular NTQ group $N$,  denote this branch by $b$  and the system corresponding to $N$  by $S(b)=1$. ($S(b):S_1(X_1,\ldots ,X_n)=1,\ldots ,S_n(X_n)=1$). Then $N=F_{R(S(b))}$
 For this fundamental
sequence $Var_{\rm
fund}(S(b))$ and regular NTQ group  $N=N_0=\bar N_0=F_{R(S(b))}$ we assign a vertex $w_{1}$ of the path. We draw an edge from
 vertex $w_{0}$  to each vertex $w_{1}$.

We say that a fundamental sequence is constructed modulo some subgroups of the coordinate group  if these subgroups are elliptic in the JSJ decompositions on all levels in the construction of this fundamental sequence.
\begin{df} \label{minhom} Let $S=1$ be a canonical NTQ system corresponding to a well-aligned fundamental sequence for a group $L$ relative to  limit groups  $L_1,\ldots ,L_k$:

\medskip
$S_1(X_1, X_2, \ldots, X_n) = 1,$

\medskip
$\ \ \ \ \ S_2(X_2, \ldots, X_n) = 1,$

$\ \ \ \ \ \ \ \ \ \  \ldots$

\medskip
$\ \ \ \ \ \ \ \ \ \ \ \ \ \ \ \ S_n(X_n) = 1$

 We say that values of $X_2,\ldots, X_n$ are {\em minimal with respect to the images of $L$} if for each $i>1$
values of $X_i,\ldots ,X_n$ are minimal in fundamental sequences for the groups $F_{R(S_i,\ldots, S_n)}$ modulo rigid subgroups and edge groups of the decomposition of the image of $L$ on level $i-1$ from the top. In particular, values of $X_2,\ldots, X_n$
are minimal in fundamental sequences for $F_{R(X_2,\ldots, X_n)}$ modulo rigid subgroups and edge groups of the image of $L$ in $F_{R(X_1,\ldots, X_n)}$.
\end{df}

\subsection {First step} Let $U_1(X_1,\ldots, X_n)=1$ be the system of equations satisfied by the images of $X_1,\ldots , X_n$ in $M$. If $F_{R(U_1)}$ is not a proper quotient of $N_0=F_{S(b)}$ we stop. Suppose it is a proper quotient.
Let $Var_{\rm
fund}(U_1)$ be the subset of homomorphisms from the set $Var_{\rm
fund}(S(b))$  minimal with respect to the canonical automorphisms modulo the images of $G$ and satisfying the additional equation $U_1(X_1,\ldots ,X _n)=1.$

 Let $G_1$ be a fully residually free group discriminated by the set of homomorphisms
$Var_{\rm fund}(U_1)$. $G_1=F_{R(U_1)}$  and we can suppose that the system $U_1(X_1,\ldots ,X_n)=1$ is  irreducible. Consider the family  of  well-aligned fundamental
sequences for $G_1$  modulo the images $R_1,\ldots, R_s$ of the factors in the free  decomposition of the subgroup $H_1=\langle X_2,\ldots, X_n\rangle $.  Since we only consider well aligned canonical fundamental sequences $c$ for $G_1$ modulo $R_1,\ldots, R_s$
they have, in particular, the following property:   $c$ is consistent with the decompositions of surfaces corresponding to
 quadratic equations of $S_1$  by a collection of simple closed curves mapped to the identity by $\pi _1$. Namely, if we refine
 the JSJ decomposition of $G$ by adding splittings
corresponding to the simple closed curves that are mapped to the identity when $G$ is mapped to the free product in \cite[Section 3.1]{KhaMya2012quantifier} then the boundary elements
of QH subgroups in this new decomposition are mapped to elliptic elements on all the levels of $c$.  Let $k_1$ be the sum of the ranks of the maximal free groups that can be the images of  the closed surfaces after we refined the JSJ decomposition.

The group  $G$ is embedded in the NTQ group corresponding to the fundamental sequence $c$.

Suppose the fundamental sequence $c$ has the top dimension component
$k_1$. If the NTQ system corresponding to the top level of the sequence
$c$ is the same as $S_1=1$, we extend the fundamental sequences
modulo $R_1,\ldots ,R_s$ by canonical fundamental sequences  for
$H_1$  modulo  the factors in the free decomposition of the subgroup
$\langle X_3,\ldots ,X_n\rangle $. Since the fundamental sequence is well alligned it has dimension less or  equal to $k_2$.  We continue this way to construct
fundamental sequences $Var_{\rm fund}(S_1(b))$.

 Suppose now that the fundamental sequence $c$ for $G_1$ modulo
$R_1,\ldots ,R_s$ has dimension strictly less than $k_1$ or has
dimension $k_1$, but the NTQ system corresponding to the top level of
$c$ is not the same as $S_1=1$.  Then we
use the following lemma (in which we suppose that $R_1,\ldots ,R_s$
are non-trivial).

\begin{lm}\label{prop}The image $G_{t}$ of  $G$ in
the group $H_{t}$ appearing on the terminal level $t$ of the sequence $c$ is a
proper quotient of $G$ unless $G$ is a free group.
\end{lm}
\begin{proof} Consider the terminal group of $c$; denote it $H_t$.
Suppose $G_{t}$ is isomorphic to $G$. Denote the abelian JSJ
decomposition of $H_t$ by $D_t$. Then there is an abelian
decomposition of $G$ induced by $D_t$. Therefore rigid (non-abelian
and non-QH) subgroups and edge groups of
$G$ are elliptic in this decomposition.
 But this is impossible because this  means that
the homomorphisms we are considering can be shortened by applying
canonical
automorphisms of $H_t$ modulo those subgroups $\{R_1,\ldots ,R_s\}$  but  $V_{fund}(U_1)$  contains only homomorphisms minimal  in fundamental sequences modulo  $R_1,\ldots ,R_s$ (see Def. \ref{minhom}).

\end{proof}

Therefore the image  of $G$ on all the levels of the fundamental sequence $c$ above some level $p$ is isomorphic to $G$ and on level $p$ is a proper quotient of $G$.    Denote the complete set  of fundamental sequences that encode all the homomorphisms  $G_p\rightarrow F$ by $\mathcal F$. One can extract from $c$ modulo level $p$
the induced  well aligned fundamental sequence for $G$,
see the definition of the induced fundamental sequence in \cite{KhaMya2006elementary}, \cite{KhaMya2012quantifier}. Denote this induced fundamental sequence up to level $p$ by $c_2$. Consider a fundamental sequence $c_3$ that consists of homomorphisms obtained by the composition of a homomorphism from $c_2$ and from a fundamental sequence
$b_2\in \mathcal F.$ Let $\bar N_1$ be the  NTQ group corresponding to $c_3$.
Then $\bar N_1$  is a regular NTQ group. The existence of the groups occurring in the fundamental sequence $c_3$ can be described by first order formulas and, therefore, there  is a quotient of $\bar N_1$ inside $M$.

Denote by $M(X,Z_1)$ (where $Z_1$ is the set of generators, $X\subset Z_1$) the group generated by  the top $p$ levels of the  NTQ group corresponding
 to the fundamental sequence $c$.  Consider the {\em block-NTQ group}  $N_1$ that is a fully residually free quotient
 of the amalgamated product of $M(X,Z_1)$  and $\bar N_1$  amalgamated  along the top $p$ levels of $\bar N_{1}$. To obtain a fully residually free quotient we   add relations to make it commutative transitive.  Assign the  sequence $c_3$, regular NTQ group $\bar N_1$, block-NTQ group $N_1$, and $M(X,Z_1)$ to the vertex $w_{2}$ of the path.
 We draw an edge from the vertex $w_{1}$  to $w_{2}$.

If there are no additional equations  on generators of $\bar N_1$ in $M$ we stop because in this case $\bar N_1$ is the regular NTQ system that is contained in $M$ and contains $G$. Suppose there are additional equations on generators of  $\bar N_1$.  Let $U_2=1$ be a system of all equations on generators of $\bar N_1$.  Moreover, we suppose that $G_2=F_{R(U_2)}$ is discriminated by homomorphisms that are minimal for $\bar N_1$ with respect to images of $G$ and minimal  for
$M(X,Z_1)$ with respect to images of $F_{R(U_1)}$.

\subsection{Second step}\label{2nd}
We will describe the next  step in the construction
which basically is general. The fundamental sequence and the block-NTQ group obtained on the second step will be assigned to vertex $w_{3}$.

Suppose  the JSJ decomposition for the NTQ system corresponding to the top level of $c$ corresponds to
the equation $S_{11}(X_{11},X_{12},\ldots )=1;$ some of the
variables $X_{11}$ are quadratic, the others correspond to
extensions of centralizers.
 Construct a canonical fundamental sequence $c^{(2)}$  for $G_2$ modulo  the factors in the free decomposition of the subgroup generated by
$X_{12},\ldots  $.

Denote by $N_0^1$ the image of the subgroup generated by $X_1,\ldots
,X_n$ in the group $M(X,Z_1)$ discriminated by $c$, we will write $N_0^1= \langle X_1,\ldots
,X_n\rangle _c.$ Denote by $N_0^2=\langle X_1,\ldots ,X_n\rangle _{c^{(2)}}$ the image of
$\langle X_1,\ldots ,X_n\rangle $
 in the group discriminated by $c^{(2)}$.
$N_0^1$ must be isomorphic to
$N_0^2$ because they correspond to the same subgroup of $M$.

{\bf Case 1.} If the top levels of $c$ and $c^{(2)}$ are the same, then
we go to the second level of $c$ and consider it the same way as the
first level.

{\bf Case 2.}  If the top levels of the NTQ system for $c$ and $S_1$ are the same
(therefore $c$ has only one level). We work with $c^{(2)}$ the same
way as we did for $c$. Suppose $c^{(2)}$ is not the same as $c$.   Then  the image of $G$ on some level
$p$ of $c^{(2)}$ is a proper quotient of $G$ by
Lemma \ref{prop}. Let $M(X,Z_2)$ be the NTQ group corresponding to the top $p$ levels of $c^{(2)}$. We consider fundamental sequences constructed as follows: the top part is the fundamental sequence induced by the top part of $c^{(2)}$
above level $p$ for $G$, and the bottom part is a fundamental sequence for this quotient of $G$ (solutions will go along the first fundamental sequence from the top level to level $p$ and then continue along one of the fundamental sequences for the image of $G$ on level $k$).  We assign each of these fundamental sequences to a vertex  $w_{3}$, then take the corresponding canonical NTQ group $\bar N_2$. Then we construct the block-NTQ group as we did on the first step, denote it by $N_2$. We also assign $N_2$ and $\bar N_2$ to the vertex  $w_{3}$.

{\bf Case 3.}  If the top levels of the NTQ system for $c$ and $S_1$ are  not  the same and
the top levels of $c$ and $c^{(2)}$ are not the same, then we look at  $N_0^2 =N_0^1$.  It follows from the minimality restriction made for the solutions used to construct $F_{R(U_2)}$ that there is some
 level $k$ from the top of $c^{(2)}$ such that we can suppose that the image of
either $G$ or $N_0^1$ on this level is a proper quotient of it (and on the levels above k is isomorphic to $G$ (resp. to $N_0^1$).

If the image of $G$ that we denote $G_t$ is a proper quotient of $G$ we do what we did on step 1. Namely, we construct the fundamental sequence for  $G$ induced from $c^{(2)}$ and continue it with a canonical fundamental sequence for $G$, construct the NTQ group $\bar N_2$ for this fundamental sequence and the block NTQ group $N_2$ and assign them to a vertex $w_3$.

Suppose now that $G_t$ is isomorphic to $G$ and $N_{0,t}^1$ is a proper quotient of $N_0^1$. Consider
fundamental sequences for $N_{0,t}^1$ modulo the images of subgroups $R_1,\ldots, R_s$, and apply to them step 1. Denote the obtained fundamental sequences  by $f_i$.
Construct fundamental sequences for the subgroup generated by the images of
$X_1,\ldots ,X_n$ with the top part being induced from the top
part of $c^{(2)}$ (above level k) and bottom part being some $f_i$,
but not the sequence with the same top part as $c$. We construct the block-NTQ group amalgamating the top $k$ levels of $c^{(2)}$ and the block-NTQ group constructed for $f_i$ as on the first step. We denote this block-NTQ group (consisting of three blocks) by $N_2$ and its subgroup that is a regular NTQ group by $\bar N_2$. These groups  are assigned then to $w_3.$

Since  we started this step assuming that $U_2=1$ is a proper equation on the generators of $\bar N_1$  It must be some level $s$ of $c$ such that we do not have cases 2 or 3 on the levels above $s$ but have one of these cases on level $s$.

\subsection{General step}\label{nth}
We now describe  the $n$'th step of the construction. Denote by
$N_i$ the block-NTQ group constructed on the $i$'th step and assigned to vertex $w_{i+1}$ and by
$N_i^j, j> i$ its image  on the $j$'th step. Denote by $\bar N_i$ the NTQ groups constructed on $i$'th step.

Let $\{j_k,\
k=1,\ldots ,s\}$ be all the indices for which the top level of
$N_{j_k+1}$ is different from the top level of $N_{j_k}$. Let $M(X, Z_i)$ be the group corresponding to the top block on step $i$ and $M^j(X, Z_i)$ its image in $N_{j-1}^j.$

On each step $i$ we consider fundamental sequences for the proper quotient of $N_{i-1}$
 modulo the images of freely indecomposable free factors  of the second level block NTQ group $N_{i-1}$ (images of rigid subgroups in the relative decomposition of $M(X, Z_{i-1})$). Let $R^{r(i)}$ be the family of these rigid subgroups on step $i$. Every time when we increase  parameter subgroups $R^{r(i)}$, $r(i)$ increases by 1.  Let $g(t)$ be the last step  when  we constructed fundamental sequences modulo $R^t$. On step $g(t)+1$ parameters subgroups are increased to become $R^{t+1}$ and fundamental sequences for the top block are constructed modulo $R^{t+1}$, so $r(g(t)+1)=t+1.$

Notice that  $G$ and all the groups $M^{g(r)}(X, Z_{g(r)})$, $g(r)<n,$ are
embedded into $N_{n-1}^n.$

{\bf Case 1.} The top levels of $c^{(n)}$ and $c^{(n-1)}$ are the
same. In this case we go to the second level and consider it the
same way as the first level.

{\bf Case 2.} The top levels of $c^{(n-1)}, c^{(n-2)},\ldots ,
c^{(n-i)}$ are the same, and the top levels of $c^{(n-1)}$ and $c^{(n)}$
are not the same. Then on some level $p$ of the NTQ group for $c^{(n)}$ we can suppose
that the image of $M^n(X,Z_{n-1})$ is a proper quotient of it (and on the levels above $p$ it is  isomorphic to $M^n(X,Z_{n-1})$). Let $r$ be the minimal such index that  $M^n(X,Z_{g(r)})_t$ is a proper quotient of $M^n(X,Z_{g(r)})=M^{g(r)}(X,Z_{g(r)})$. Then the levels below the first level (level 2 and below) of the block NTQ group $N_n$ that we are constructing on step $n$  will correspond to the block NTQ group $\tilde N_{g(r)+1}$ that we would construct on step $g(t)+1$ for $M^n(X,Z_{g(r)})_t$.  Moreover, we consider only fundamental sequences for $\tilde N_{g(r)+1}$ with the top level different from $c^{(g(r))}$ (not of maximal complexity).

{\bf Case 3.} The top levels of $c^{(n-2)}$ and $c^{(n-1)}$ are not
the same and the top levels of $c^{(n-1)}$ and $c^{(n)}$ are not the
same. Then on some level $p$ of $c^{(n)}$ the image of
$M^n(X,Z_{n-1})$  is a proper quotient of $M^n(X,Z_{n-1}).$  Let $r$ be the minimal such index that  $M^n(X,Z_{g(r)})_t$ is a proper quotient of $M^n(X,Z_{g(r)})=M^{g(r)}(X,Z_{g(r)})$. Then the levels below the first level (level 2 and below) of the block NTQ group $N_n$ that we are constructing on step $n$  will correspond to the block NTQ group $\tilde N_{g(r)+1}$ that we would construct on step $g(t)+1$ for $M^n(X,Z_{g(r)})_t$.  Moreover, we consider only fundamental sequences for $\tilde N_{g(r)+1}$ with the top level different from $c^{(g(r))}$ (not of maximal complexity).

If going from the top to the bottom of the block-NTQ system, on some level of the block-NTQ group we can apply Case 2 or 3 to this level we apply it. There will be always such a level.

\subsection{Some auxiliary results}
We recall several results that we will need.

Let $G$ be a freely indecomposable limit group with the JSJ decomposition $D$. Let $K$ be the fundamental group of the subgraph  of  groups  generated by the rigid subgroups of $D$. Let  $\phi$ be an epimorphism from $G$ to a limit group $H$, and suppose that $H$  is freely indecomposable relative to $\phi (K)$ and $\Delta$ is the JSJ decomposition of $H$ relative to $\phi (K)$. Suppose that $\phi$ is non-pinching on each QH subgroup.

Let $Q$ be a QH subgroup in $D$ and let $S$ be the corresponding punctured surface. Since the boundary elements of $Q$ are mapped by $\phi $ to elliptic elements in $\Delta$, the (maybe trivial) cyclic decomposition induced on $\phi (Q)$ by $\Delta$ can be lifted to (maybe trivial) maximal cyclic decomposition of $Q$, in which every cyclic edge group is mapped by $\phi $ to an elliptic element in $\Delta$ which corresponds to some decomposition of the punctured surface $S$ along a maximal collection of disjoint non-homotopic simple closed curves. Let $\Delta (Q)$ be the corresponding cyclic decomposition of $Q$ and let $\Delta (S)$ be the associated  collection of simple closed curves.

\begin{lm} \label{11}
Let $\bar Q$ be a QH subgroup in $\Delta$, and let $\bar S$ be the associated punctured surface. If $\phi$ maps non-trivially a connected subsurface $S_1$ of $S - \Delta (S)$ into $\bar S$, then $genus (\bar S)\leq genus (S)$ and $|\chi (\bar S)|\leq |\chi (S)|$. Moreover, in this case $\phi$ maps the fundamental group of the subsurface $S_1$  into a finite index subgroup of $\bar Q$.
\end{lm}
\begin{proof} Since the boundary components of $S_1$ are mapped to non-trivial elliptic elements of in $\Delta$,  we have  $genus (\bar S)\leq genus (S)$ and $|\chi (\bar S)|\leq |\chi (S)|$.  The last statement of the lemma follows from \cite[Lemma 7]{KhaMya2006elementary}.
\end{proof}

In the case when there is a pinching corresponding to $\phi$
 we construct another graph of groups $\Gamma$ obtained from $D$  by cutting the surfaces corresponding to QH vertex groups along the pinching and  filling each curve with a disk. Then $\phi=\phi _1\circ\theta$,  where $\theta$ maps the pinching to the identity and $\phi _1$ is non-pinching.

We call a QH subgroup of $D$ {\em stable} if there exists a QH subgroup $\bar Q$ in $\Delta$ so that $\phi _2$ maps $Q$ into $\bar Q$ and the sizes of $Q$ and $\bar Q$ are the same.

We now give the setting for the next lemma.  Let a regular  NTQ system $S(X_1,\ldots, X_n)=1$ have the form
$$S_1(X_1,\ldots ,X_n)=1,$$
$$\ldots $$
$$S_n(X_n)=1,$$
where $S_1=1$ corresponds to  the top level of the  JSJ decomposition  for
$\Gamma _{R(S)}$, variables from $X_1$ are quadratic. Consider this system together with the
fundamental sequence $V_{\rm fund}(S)$ defining it. Let $V_{\rm
fund}(U_1)$ be the subset of $V_{\rm fund}(S)$ satisfying some
additional equation $U_1=1$, and $G_1$ a group discriminated by this
subset. Consider the family  of those canonical fundamental
sequences for $G_1$ modulo the images $R_1,\ldots ,R_s$ of the
factors $P_1,\ldots ,P_s$ in the free decomposition of
 the subgroup $\langle X_2,\ldots ,X_m\rangle $, which have the same Kurosh rank modulo them as
 $S_1=1$ and are compatible up to an automorphism with the pinching of QH subgroups of $F_{R(S)}$.  Denote this free decomposition by $H_1*$. The terminal group of such a fundamental sequence does not have a sufficient splitting (see \cite{KhaMya2006elementary}) relative to  $H_1*$.

Denote such a fundamental sequence by $c$, and the corresponding NTQ
system $\bar S=1 (mod \ H_1*)$, where $\bar S=1$ has the form
$$\bar S_1(X_{11},\ldots ,X_{1m})=1$$
$$\ldots $$
$$\bar S_m(X_{1m})=1.$$

Denote by $D_S$ a canonical decomposition corresponding to the group
$F _{R(S)}$. Non-QH subgroups in this decomposition are
$P_1,\ldots ,P_s$, QH subgroups correspond to the system
$S_1(X_1,\ldots ,X_n)=1.$ For each $i$ there exists a canonical
homomorphism
$$\eta _i: F _{R(S)}\rightarrow F _{R(\bar S_i,\ldots ,\bar S_m)}$$
such that $P_1,\ldots ,P_s$ are mapped into
rigid subgroups in the canonical decomposition of $\eta
_i(F _{R(S)})$.

Each QH subgroup in the decomposition  of $F _{R(\bar S_i,\ldots ,\bar S_m)}$
as an NTQ group is a QH subgroup of $\eta _i(F _{R(S)})$. Since $\eta _i$ can be represented as a composition of the map killing the pinching on QH subgroups of $D_S$ and a non-pinching map,  by Lemma \ref{11},
for each QH subgroup $Q_1$ of $\eta
_i(F _{R(S)})$ there exists a QH subgroup of the image of $F _{R(S)}$ after killing the pinching (not necessarily a maximal QH subgroup) that is
mapped into a subgroup of finite index in $Q_1$. The size of this QH
subgroup is, obviously, greater or equal to the size of $Q_1$. Those
QH subgroups of $F _{R(S)}$ that are mapped into QH subgroups of the
same size by some $\eta _i$ are  stable.

\begin{lm}\label{ch}
In the conditions above there are the following possibilities:

(i) $size\ (\bar S) = size\ (S_1)$,  in this case $c$ has only one level identical to $S_1$;

(ii) It is possible to modify the system $\bar S=1$ in such a way that $size\ (\bar S) <
size\ (S_1)$;

\end{lm}

\begin{proof}

 The fundamental sequence $c$ modulo the decomposition $H_1* $
has the same Kurosh rank as $S_1=1$. The Kurosh rank of $Q_1=1$ is the
sum of the following  numbers: \begin{enumerate} \item [1)] the
dimension of a free factor $F_1$ in the free
decomposition of $F_{R(S)}$ corresponding to an empty equation in
$S_1=1$;
\item [2)] the sum of dimensions of surface group factors; \item
[3)] the number of free variables of quadratic equations with
coefficients in $S_1=1$ corresponding to the fundamental sequence
$Var _{\rm fund}(S)$, \item [4)] $factors (Var_{\rm fund})$.\end{enumerate} Because $c$ has the same
Kurosh rank, the free factor $F_1$ is unchanged. Surface factors are sent into
different free factors.

If $size\ (\bar S) = size\ (S_1)$, then the generic family of solutions for $S_1=1$ factors through $\bar S=1$ and in this case $c$ has only one level identical to $S_1$.

Otherwise,  by the analog of \cite[Theorem 9]{KhaMya2005implicit} we move all the stable QH subgroups corresponding to the system $S_1=1$ to the bottom level $m$ of  $\bar S$. If all QH subgroups corresponding to the system $S_1=1$ were stable we would have
$size\ (\bar S) = size\ (S_1)$.  Therefore there are non-stable QH subgroups. Then
$size\ (\bar S) < size\ (S_1)$ by Lemma \ref{11} and the paragraph after it.
The lemma  is proved.
\end{proof}

\subsection{The  procedure is finite}

In this subsection we will prove the following result.

\begin{prop} \label{AE} The path $w_0,w_1,...,$ is finite.\end{prop}

 We will use induction on the complexity of $S_1(X_1,\ldots , X_n)=1$.  The induction assumption is that for any block-NTQ system  for which the complexity of the first level is less than the complexity of $S_1(X_1,\ldots , X_n)=1$ the procedure is finite. Suppose the procedure  is infinite.

By Lemma \ref{ch},
every time we apply the transformation of Case 3 (we refer to the
cases from Section \ref{nth}) in the construction  we
either (i) decrease the dimension in the top block, therefore decrease the Kurosh rank, or (ii) replace
the NTQ system in the top block  by another NTQ system of the same
dimension but of a smaller size. Hence the complexity defined in the beginning of this section decreases. Notice that the complexity of the top block is bounded by the complexity of $S_1(X_1,\ldots ,S_n)=1.$ Hence,
Case 3 cannot be applied infinitely many times to the top block.

Therefore starting at some step $n_1$ for any $n>n_1$, $r(n)=r(n_1)$, the top blocks of all $N_n$ are the same and fundamental sequences modulo $R^{r(n)}=R^{r(n_1)}$  for the top block are of maximal complexity.

Therefore eventually beginning at some step after step $n_1$ we are  applying the procedure to the second block, more precisely to the terminal level of first block fundamental sequence.

If at some step $n>n_1$ $G^{(n)}_t$ is a proper quotient of $G$ then  the procedure for the second level is eventually applied to $G^{(n)}_t$ and  is finite by induction. Otherwise, if
we apply Case 2, we consider the second block for  proper quotients
of a finite number of groups.  Let $r<r(n_1)$ be the minimal index for which $N^n(X, Z_{g(r)})_t$ is a proper quotient of $N^{g(r)}(X, Z_{g(r)})$ for some $n>n_1.$  Moreover we only consider fundamental sequences for $N^n(X, Z_{g(r)})_t$ modulo $R^{r+1}$  that are not of maximal complexity, in particular, their complexity is less than the complexity of
$S_1(X_1,\ldots , X_n)=1$. By induction, the procedure for $N^n(X, Z_{g(r)})_t$ is finite.

This proves the proposition. Theorem \ref{main} follows from the proposition.



\begin{thebibliography}{99}
\bibitem{KhaMya2005effective} O. Kharlampovich and A. Myasnikov. ÒEffective JSJ decompositionsÓ. In: Groups, languages, algorithms. Ed. by A. Borovik. Vol. 378. Contemporary Mathematics, 2005.
\bibitem{KhaMya2005implicit} O. Kharlampovich and A. Myasnikov. ÒImplicit function theorem over free groupsÓ. In: Journal of Algebra 290.1 (2005), pp. 1-203.
\bibitem{KhaMya2006elementary} O. Kharlampovich and A. Myasnikov. ÒElementary theory of free non-abelian groupsÓ. In: Journal of Algebra 302.2 (2006), pp. 451-552.
\bibitem{KhaMya2010equations}  O. Kharlampovich and A. Myasnikov. ÒEquations and fully residually free groupsÓ. In: Combinatorial and geometric group theory. Springer, 2010, pp. 203-242.
\bibitem{KhaMya2012quantifier} O. Kharlampovich and A. Myasnikov. Quantifier elimination algorithm to boolean combination of $\exists\forall$-formulas in the theory of a free group, In "Groups and Model Theory, GAGTA book 2",  DeGruyter 2021.
\bibitem{KhaMyaSkl2020fraisse} O.Kharlampovich, A.Myasnikov, and R.Sklinos.ÒFraisse limits of limit groupsÓ. In: Journal of Algebra 545 (2020), pp. 300-323.
\bibitem{KhaNat2020non} Olga Kharlampovich and Christopher Natoli. Non $\forall$- homogeneity in free groups. Bulletin of Mathematical Sciences, 11, no 1, 2021. 
\bibitem{Levitt} G. Levitt. ÒAutomorphisms of Hyperbolic Groups and Graphs of GroupsÓ. In: Geometriae Dedicata 114 (2004), pp. 49-70.
\bibitem{Per2011elementary} C. Perin. ÒElementary embeddings in torsion-free hyperbolic groupsÓ. In: Annales scientifiques de l'Ecole Normale Superieure 44.4 (2011), pp. 631-681.
\bibitem{Sel2006diophantine} Z. Sela. ÒDiophantine geometry over groups VI: The elementary theory of a free groupÓ. In:  GAFA 16.3 (2006), pp. 707-730.
\end{thebibliography}
\end{document}